\documentclass{article}

\usepackage{amsmath}
\usepackage{amsthm}

\usepackage[pdftex]{hyperref}
\usepackage{alltt}

\newtheorem{theorem}{Theorem}[section]

\numberwithin{equation}{section}

\newcommand{\cycle}[2]{\genfrac{[}{]}{0pt}{}{#1}{#2}}

\title{On Two Types of Harmonic Number Identities}
\author{M.J. Kronenburg}

\begin{document}

\maketitle

\begin{abstract}
Two types of finite series of products of harmonic numbers
involving nonnegative integer powers are evaluated,
also yielding two other important harmonic number identities.
The recursion formulas for these sums are derived,
which are easily translated into a computer program.
\end{abstract}


\section{Definitions and Basic Identities}

The generalized harmonic numbers used in this paper are:
\begin{equation}
  H_n^{(m)} = \sum_{k=1}^n \frac{1}{k^m}
\end{equation}
from which follows that $H_0^{(m)}=0$.
The traditional harmonic numbers are:
\begin{equation}
 H_n=H_n^{(1)}
\end{equation}
A well known identity is \cite{GKP94,K97,SLL89,S90}:
\begin{equation}\label{sum1okhk}
 \sum_{k=1}^n \frac{1}{k} H_k = \frac{1}{2} ( H_n^2 + H_n^{(2)} )
\end{equation}
and \cite{K97,K11}:
\begin{equation}\label{sum1okp1hk}
 \sum_{k=0}^n \frac{1}{k+1}H_k = \frac{1}{2} ( H_{n+1}^2 - H_{n+1}^{(2)} )
\end{equation}
Let $B_n^+$ be the Bernoulli number with $B_1^+=1/2$ instead of $B_1=-1/2$:
\begin{equation}
 B_n^+ = (-1)^n B_n = B_n + \delta_{n,1}
\end{equation}
Then for nonnegative integer $p$ \cite{K11}:
\begin{equation}\label{faulhaber}
 H_{n}^{(-p)} = \frac{1}{p+1}\sum_{k=1}^{p+1}\binom{p+1}{k}B_{p-k+1}^+ n^k
\end{equation}
and for nonnegative integer $p$ \cite{K11}:
\begin{equation}\label{harmsumf}
\begin{split}
 & \sum_{k=0}^n k^p H_k^{(m)} = F(n,p,m) \\
 & = H_n^{(-p)} H_n^{(m)} + H_n^{(m-p)}
  - \frac{1}{p+1} \sum_{k=1}^{p+1} \binom{p+1}{k} B_{p-k+1}^+ H_n^{(m-k)}
\end{split}
\end{equation}

When $a\leq b$ are two integers and $\{x_k\}$ and $\{y_k\}$ are two sequences of complex numbers,
and $\{s_k\}$ the sequence of complex numbers defined by:
\begin{equation}
 s_k = \sum_{i=a}^k x_i
\end{equation}
then there is the following summation by parts formula \cite{K11}:
\begin{equation}\label{sumparts}
 \sum_{k=a}^{b-1} x_k y_k = s_{b-1} y_b - \sum_{k=a}^{b-1} s_k ( y_{k+1} - y_k )
\end{equation}

\section{First Type of Harmonic Number Identities}

\begin{theorem}
For nonnegative integer $n$:
\begin{equation}\label{harmsumconv}
 \sum_{k=0}^n H_k H_{n-k} = (n+1)[(H_{n+1}-1)^2 - H_{n+1}^{(2)} + 1]
\end{equation}
\end{theorem}
\begin{proof}
With generating functions: 
let $[z^n]f(z)$ be the coefficient of $z^n$ in the power series expansion of $f(z)$.
Then \cite{GKP94}:
\begin{equation}
 H_n = [z^n] \frac{1}{1-z} \ln(\frac{1}{1-z})
\end{equation}
The generating function of a convolution of two series is the product
of their generating functions:
\begin{equation}
 \sum_{k=0}^n H_k H_{n-k} = [z^n] \frac{1}{(1-z)^2} (\ln(\frac{1}{1-z}) )^2
\end{equation}
These generating functions are known \cite{GKP94}:
\begin{equation}
 [z^n] \frac{1}{(1-z)^2} = n+1
\end{equation}
\begin{equation}
 [z^n] (\ln(\frac{1}{1-z}))^2 = 2 \cycle{n}{2} \frac{1}{n!}
\end{equation}
where $\cycle{a}{b}$ is the Stirling number of the first kind,
which in this case evaluates to \cite{GKP94}:
\begin{equation}
 \cycle{n}{2} =
 \begin{cases}
   0 & \text{if $n=0$} \\
  (n-1)! H_{n-1} & \text{if $n>0$} \\
 \end{cases}
\end{equation}
The product of these two generating functions results in another convolution:
\begin{equation}
 [z^n] \frac{1}{(1-z)^2}(\ln(\frac{1}{1-z}))^2 = 2\sum_{k=1}^n(n-k+1)\frac{1}{k}H_{k-1}
\end{equation}
which using (\ref{harmsumf}) with $p=0$ and $m=1$ and (\ref{sum1okp1hk}) results in:
\begin{equation}\label{harmsumconvx}
\begin{split}
 & 2(n+1)\sum_{k=0}^{n-1}\frac{1}{k+1}H_k - 2\sum_{k=0}^{n-1}H_k \\
 & = (n+1)[H_n^2-H_n^{(2)}] - 2[n H_n - n] \\
\end{split}
\end{equation}
With $H_n^{(m)}=H_{n+1}^{(m)}-1/(n+1)^m$ this results in the theorem.
\end{proof}
Using different methods this theorem was already proved in literature \cite{S90,WGW12}.
Now evaluating the sum directly yields:
\begin{equation}
\begin{split}
 \sum_{k=1}^n H_kH_{n-k} & = \sum_{k=1}^{n-1}\sum_{i=1}^k\frac{1}{i}\sum_{j=1}^{n-k}\frac{1}{j} \\
 & = \sum_{i=1}^{n-1}\frac{1}{i}\sum_{k=i}^{n-1}\sum_{j=1}^{n-k}\frac{1}{j} \\
 & = \sum_{i=1}^{n-1}\frac{1}{i}\sum_{j=1}^{n-i}\frac{1}{j}\sum_{k=i}^{n-j}1 \\
 & = \sum_{i=1}^{n-1}\frac{1}{i}\sum_{j=1}^{n-i}\frac{1}{j}(n-i-j+1) \\
 & = (n+1)\sum_{i=1}^{n-1}\frac{1}{i}H_{n-i} - \sum_{i=1}^{n-1}H_{n-i} - \sum_{i=1}^{n-1}\frac{n-i}{i} \\
 & = (n+1)\sum_{i=1}^{n-1}\frac{1}{i}H_{n-i} - \sum_{i=1}^{n-1}H_i - n H_{n-1} + (n-1)
\end{split}
\end{equation}
The second sum is evaluated using (\ref{harmsumf}) with $p=0$ and $m=1$,
and substituting the right side of (\ref{harmsumconvx}) results in the following identity:
\begin{equation}
 \sum_{k=1}^n\frac{1}{k}H_{n-k} = H_n^2 - H_n^{(2)}
\end{equation}
Evaluating the sum for $n+1$ and using $H_{n-k+1}=H_{n-k}+1/(n-k+1)$ yields:
\begin{equation}
\begin{split}
 \sum_{k=0}^{n+1} H_kH_{n+1-k} & = \sum_{k=0}^n H_kH_{n-k+1} \\
 & = \sum_{k=0}^n H_kH_{n-k} + \sum_{k=0}^n \frac{1}{n-k+1}H_k \\
 & = \sum_{k=0}^n H_kH_{n-k} + \sum_{k=0}^n \frac{1}{k+1} H_{n-k}
\end{split}
\end{equation}
Substituting (\ref{harmsumconv}) for $n+1$ and $n$ results in the following identity:
\begin{equation}\label{sum1okp1hnmk}
 \sum_{k=0}^n \frac{1}{k+1} H_{n-k} = H_{n+1}^2 - H_{n+1}^{(2)}
\end{equation}

Using summation by parts (\ref{sumparts}) with $x_k=H_{n-k}$ and $y_k=k^pH_k$,
the following recursion formula results:
\begin{equation}
\begin{split}
 \sum_{k=0}^n k^p H_k H_{n-k} & = \frac{1}{p+1} \{ (n+1)^p [ (n+1)H_n^2-(n-1)H_n-\frac{n}{n+1} ] \\
 & - (n+1)H_n \sum_{k=0}^n [ ((k+1)^p-k^p)H_k + (k+1)^{p-1} ] \\
 & + \sum_{k=0}^n [ (n-k)((k+1)^p-k^p)+pk^p ] H_k H_{n-k} \\
 & + \sum_{k=0}^n k [ (k+1)^p - k^p + (n-k+1)^{p-1} ] H_k \\
 & + \sum_{k=0}^n k(k+1)^{p-1} \}
\end{split}
\end{equation}
For $p=0$ this formula yields (\ref{harmsumconv}) using (\ref{sum1okp1hnmk}).
Letting $R(n,0)$ be (\ref{harmsumconv}) and using $F(n,p,m)$ from (\ref{harmsumf}) 
evaluation of this formula results in the following recursion formula for integer $p>0$:
\begin{equation}\label{harmrecurconv}
\begin{split}
 \sum_{k=0}^n k^p H_k H_{n-k} & = R(n,p) \\
 & = \frac{1}{p+1} \{ - n(n+1)^p H_{n+1} \\
 & + \sum_{k=0}^{p-1} \binom{p}{k} [ ( n - \frac{k}{p-k+1} ) R(n,k) \\
 & + ( 1 + (-1)^k \frac{p-k}{p} (n+1)^{p-k-1} ) F(n,k+1,1) \\
 & + \frac{p-k}{p} H_n^{(-k-1)} ] \}
\end{split}
\end{equation}
For a list of the resulting expressions up to $p=5$ see section (\ref{examples}) below.

\section{Second Type of Harmonic Number Identities}

\begin{theorem}
For nonnegative integer $n$:
\begin{equation}\label{harmsumsqr}
 \sum_{k=0}^n H_k^2 = (n+1)H_{n+1}^2 -(2n+3)H_{n+1} +2(n+1)
\end{equation}
\end{theorem}
\begin{proof}
\begin{equation}
\begin{split}
 \sum_{k=0}^n H_k^2 & = \sum_{k=1}^n\sum_{i=1}^k \frac{1}{i} H_k \\
 & = \sum_{i=1}^n \frac{1}{i} \sum_{k=i}^n H_k \\
 & = \sum_{i=1}^n \frac{1}{i} ( \sum_{k=1}^n H_k - \sum_{k=1}^{i-1} H_k ) \\
 & = H_n [(n+1)H_{n+1}-(n+1)] - \sum_{i=1}^n \frac{1}{i}( iH_i-i ) \\
 & = H_n [(n+1)H_{n+1}-(n+1)] - (n+1)H_{n+1} + 2n + 1 \\
 & = (n+1)H_{n+1}^2 -(2n+3)H_{n+1} +2(n+1)
\end{split}
\end{equation}
\end{proof}

Using summation by parts (\ref{sumparts}) with $x_k=H_k$ and $y_k=k^pH_k$,
the following recursion formula results:
\begin{equation}
\begin{split}
 \sum_{k=0}^n k^p H_k^2 & = \frac{1}{p+1} \{ (n+1)^p [ (n+1)H_n^2-(n-1)H_n-\frac{n}{n+1} ] \\
 & - \sum_{k=0}^n [ (k+1)((k+1)^p-k^p)-pk^p ] H_k^2 \\
 & - \sum_{k=0}^n [ (k+1)^p - k((k+1)^p-k^p) ] H_k \\
 & + \sum_{k=0}^n k(k+1)^{p-1} \}
\end{split}
\end{equation}
For $p=0$ this formula yields (\ref{harmsumsqr}).
Letting $S(n,0)$ be (\ref{harmsumsqr}) and using $F(n,p,m)$ from (\ref{harmsumf}) 
evaluation of this formula results in the following recursion formula for integer $p>0$:
\begin{equation}\label{harmrecursqr}
\begin{split}
 \sum_{k=0}^n k^p H_k^2 & = S(n,p) \\
 & = \frac{1}{p+1} \{ (n+1)[(n+1)^p(H_{n+1}^2-H_{n+1})-H_{n+1}+1] \\
 & - \sum_{k=0}^{p-1} \binom{p}{k} [ ( 1 + \frac{k}{p-k+1} ) S(n,k) \\
 & + ( \frac{p-k}{k+1} - 1 ) F(n,k+1,1) \\
 & - \frac{p-k}{p} H_n^{(-k-1)} ] \}
\end{split}
\end{equation}
For a list of the resulting expressions up to $p=5$ see section (\ref{examples}) below.\\
The following harmonic number identities can be expressed as a sum of this type:
\begin{equation}\label{harmsqr2}
\begin{split}
 \sum_{k=0}^n k^p H_{n-k}^2 & = \sum_{k=0}^n (n-k)^p H_k^2 \\
 & = \sum_{k=0}^p (-1)^k \binom{p}{k} n^{p-k} S(n,k)
\end{split}
\end{equation}
For a list of the resulting expressions up to $p=5$ see section (\ref{examples}) below.

\section{Examples}\label{examples}

\begin{equation}
 \sum_{k=1}^n \frac{1}{k} H_k = \frac{1}{2} (H_n^2 + H_n^{(2)})
\end{equation}
\begin{equation}
 \sum_{k=0}^n \frac{1}{k+1} H_k = \frac{1}{2} (H_{n+1}^2 - H_{n+1}^{(2)})
\end{equation}
\begin{equation}
 \sum_{k=1}^n\frac{1}{k}H_{n-k} = H_n^2 - H_n^{(2)}
\end{equation}
\begin{equation}
 \sum_{k=0}^n \frac{1}{k+1} H_{n-k} = H_{n+1}^2 - H_{n+1}^{(2)}
\end{equation}
Formula (\ref{faulhaber}):
\begin{equation}
 H_n^{(0)} = n
\end{equation}
\begin{equation}
 H_n^{(-1)} = \frac{1}{2}n(n+1)
\end{equation}
\begin{equation}
 H_n^{(-2)} = \frac{1}{6}n(n+1)(2n+1)
\end{equation}
\begin{equation}
 H_n^{(-3)} = \frac{1}{4}n^2(n+1)^2
\end{equation}
\begin{equation}
 H_n^{(-4)} = \frac{1}{30}n(n+1)(2n+1)(3n^2+3n-1)
\end{equation}
\begin{equation}
 H_n^{(-5)} = \frac{1}{12}n^2(n+1)^2(2n^2+2n-1)
\end{equation}
Formula (\ref{harmrecurconv}):
\begin{equation}
 \sum_{k=0}^n H_k H_{n-k} = (n+1)[(H_{n+1}-1)^2 - H_{n+1}^{(2)} + 1]
\end{equation}
\begin{equation}
 \sum_{k=0}^n k H_k H_{n-k} = H_n^{(-1)}[(H_{n+1}-1)^2 - H_{n+1}^{(2)} + 1]
\end{equation}
\begin{equation}
\begin{split}
 \sum_{k=0}^n k^2 H_k H_{n-k} & = H_n^{(-2)}[H_{n+1}^2 - H_{n+1}^{(2)}]
  - \frac{1}{18}n(n+1)(13n+5) H_{n+1} \\
 & + \frac{1}{108}n(n+1)(71n+37)
\end{split}
\end{equation}
\begin{equation}
\begin{split}
 \sum_{k=0}^n k^3 H_k H_{n-k} & = H_n^{(-3)}[H_{n+1}^2 - H_{n+1}^{(2)}]
  - \frac{1}{12}n^2(n+1)(7n+5) H_{n+1} \\
 & + \frac{1}{72}n^2(n+1)(35n+37)
\end{split}
\end{equation}
\begin{equation}
\begin{split}
 \sum_{k=0}^n k^4 H_k H_{n-k} & = H_n^{(-4)}[H_{n+1}^2 - H_{n+1}^{(2)}] \\
 & - \frac{1}{900}n(n+1)(447n^3+468n^2+17n-32) H_{n+1} \\
 & + \frac{1}{54000}n(n+1)(20739n^3+33066n^2+4129n-3934)
\end{split}
\end{equation}
\begin{equation}
\begin{split}
 \sum_{k=0}^n k^5 H_k H_{n-k} & = H_n^{(-5)}[H_{n+1}^2 - H_{n+1}^{(2)}] \\
 & - \frac{1}{360}n^2(n+1)(157n^3+218n^2+17n-32) H_{n+1} \\
 & + \frac{1}{21600}n^2(n+1)(6839n^3+14566n^2+4129n-3934)
\end{split}
\end{equation}
Formula (\ref{harmrecursqr}):
\begin{equation}
 \sum_{k=0}^n H_k^2 = (n+1)H_{n+1}^2 - (2n+3)H_{n+1} + 2(n+1)
\end{equation}
\begin{equation}
 \sum_{k=0}^n k H_k^2 = H_n^{(-1)}H_{n+1}^2 - \frac{1}{2}(n^2+n-1)H_{n+1} + \frac{1}{4}(n+1)(n-2)
\end{equation}
\begin{equation}
 \sum_{k=0}^n k^2 H_k^2 = H_n^{(-2)}H_{n+1}^2 - \frac{1}{18}(4n^3+9n^2+5n+3)H_{n+1} + \frac{1}{108}(n+1)(8n^2+n+18)
\end{equation}
\begin{equation}
\begin{split}
 \sum_{k=0}^n k^3 H_k^2 & = H_n^{(-3)}H_{n+1}^2 - \frac{1}{24}n(n+1)(n+2)(3n+1)H_{n+1} \\
 & + \frac{1}{288}n(n+1)(9n^2+13n+14)
\end{split}
\end{equation}
\begin{equation}
\begin{split}
 \sum_{k=0}^n k^4 H_k^2 & = H_n^{(-4)}H_{n+1}^2 \\
 & - \frac{1}{900}(72n^5+315n^4+410n^3+135n^2-32n-30)H_{n+1} \\
 & + \frac{1}{54000}(n+1)(864n^4+2241n^3+2629n^2+1466n-1800)
\end{split}
\end{equation}
\begin{equation}
\begin{split}
 \sum_{k=0}^n k^5 H_k^2 & = H_n^{(-5)}H_{n+1}^2 \\
 & - \frac{1}{360}n(n+1)(n+2)(2n+1)(10n^2+19n-9)H_{n+1} \\
 & + \frac{1}{21600}n(n+1)(200n^4+736n^3+1159n^2+971n-366)
\end{split}
\end{equation}
Formula (\ref{harmsqr2}):
\begin{equation}
 \sum_{k=0}^n H_{n-k}^2 = (n+1)H_{n+1}^2 - (2n+3)H_{n+1} + 2(n+1)
\end{equation}
\begin{equation}
 \sum_{k=0}^n k H_{n-k}^2 = H_n^{(-1)}H_{n+1}^2 - \frac{1}{2}(3n^2+5n+1)H_{n+1} + \frac{1}{4}(n+1)(7n+2)
\end{equation}
\begin{equation}
\begin{split}
 \sum_{k=0}^n k^2 H_{n-k}^2 & = H_n^{(-2)}H_{n+1}^2 - \frac{1}{18}(2n+1)(11n^2+17n+3)H_{n+1} \\
 & + \frac{1}{108}(n+1)(170n^2+109n+18)
\end{split}
\end{equation}
\begin{equation}
\begin{split}
 \sum_{k=0}^n k^3 H_{n-k}^2 & = H_n^{(-3)}H_{n+1}^2 - \frac{1}{24}n(n+1)(25n^2+37n+10)H_{n+1} \\
 & + \frac{1}{288}n(n+1)(415n^2+427n+130)
\end{split}
\end{equation}
\begin{equation}
\begin{split}
 & \sum_{k=0}^n k^4 H_{n-k}^2 = H_n^{(-4)}H_{n+1}^2 \\
 & - \frac{1}{900}(2n+1)(411n^4+1002n^3+679n^2+28n-30)H_{n+1} \\
 & + \frac{1}{54000}(n+1)(72114n^4+103491n^3+46129n^2+1466n-1800)
\end{split}
\end{equation}
\begin{equation}
\begin{split}
 & \sum_{k=0}^n k^5 H_{n-k}^2 = H_n^{(-5)}H_{n+1}^2 \\
 & - \frac{1}{120}n(n+1)(98n^4+236n^3+159n^2+n-14)H_{n+1} \\
 & + \frac{1}{21600}n(n+1)(26978n^4+49996n^3+29599n^2+1961n-3234)
\end{split}
\end{equation}

\section{Computer Program}

The Mathematica$^{\textregistered}$ \cite{W03} program used to compute the expressions
is given below, which should be added to the program listed in \cite{K11}.

\begin{alltt}
HarmSumR[0]=(n+1)((HarmonicNumber[n+1]-1)^2-HarmonicNumber[n+1,2]+1);
HarmSumR[p_]:=HarmSumR[p]=Simplify[1/(p+1)
 (-n(n+1)^p HarmonicNumber[n+1]
 +Sum[Binomial[p,k]((n-k/(p-k+1))HarmSumR[k]
 +(1+(-1)^k(p-k)/p(n+1)^(p-k-1))HarmSumF[k+1,1]
 +(p-k)/p HarmFun[-k-1]),\{k,0,p-1\}])]
HarmonicSumR[p_]:=Module[\{t=HarmTable[2],u\},
 u=Factor[CoefficientArrays[HarmSumR[p],t]];
 u[[1]]+Dot[u[[2]],t]+Dot[Dot[u[[3]],t],t]]
HarmSumS[0]=(n+1)HarmonicNumber[n+1]^2-(2n+3)HarmonicNumber[n+1]
 +2(n+1);
HarmSumS[p_]:=HarmSumS[p]=Simplify[1/(p+1)
 ((n+1)((n+1)^p(HarmonicNumber[n+1]^2-HarmonicNumber[n+1])
  -HarmonicNumber[n+1]+1)
 -Sum[Binomial[p,k]((1+k/(p-k+1))HarmSumS[k]
 +((p-k)/(k+1)-1)HarmSumF[k+1,1]
 -(p-k)/p HarmFun[-k-1]),\{k,0,p-1\}])]
HarmonicSumS[p_]:=Module[\{t=HarmTable[1],u\},
 u=Factor[CoefficientArrays[HarmSumS[p],t]];
 u[[1]]+Dot[u[[2]],t]+Dot[Dot[u[[3]],t],t]]
HarmSumT[p_]:=Simplify[
 Sum[(-1)^k Binomial[p,k]n^(p-k)HarmSumS[k],\{k,0,p\}]]
HarmonicSumT[p_]:=Module[\{t=HarmTable[1],u\},
 u=Factor[CoefficientArrays[HarmSumT[p],t]];
 u[[1]]+Dot[u[[2]],t]+Dot[Dot[u[[3]],t],t]]

(* Compute some examples *)
HarmonicSumR[3]//TraditionalForm
HarmonicSumS[4]//TraditionalForm
HarmonicSumT[2]//TraditionalForm
\end{alltt}

\pdfbookmark[0]{References}{}

\end{document}